\documentclass[letterpaper, 10 pt, conference]{ieeeconf}  

\IEEEoverridecommandlockouts                              
\usepackage{microtype}
\usepackage{overpic}
\usepackage{booktabs} 

\usepackage{amssymb,latexsym,amsfonts,amsmath}
\usepackage{graphicx}
\usepackage{paralist}
\usepackage{capt-of}
\usepackage{xcolor}
\usepackage{dsfont}
\usepackage{tikz}
\usepackage[export]{adjustbox}
\usepackage{tcolorbox}
\usepackage{chngcntr}
\usepackage[normalem]{ulem}
\usepackage{extarrows}

\usetikzlibrary{calc,shapes,arrows,intersections,automata,shapes,chains}
\usetikzlibrary{shapes.multipart}
\usepackage{amssymb,latexsym,amsfonts,amsmath}
\usepackage{paralist}
\usepackage{cite}
\usepackage{eso-pic}
\usepackage{epsfig}
\usepackage{mathtools}
\usepackage{epstopdf}
\usepackage{graphicx}
\usepackage{epsfig}
\usepackage{mathtools}
\usepackage{amsfonts}
\usepackage{lineno}
\usepackage{tikz}
\usetikzlibrary{positioning,shapes}
\usetikzlibrary{arrows}
\usepackage{caption}
\allowdisplaybreaks

\newtheorem{theorem}{Theorem}
\newtheorem{definition}[theorem]{Definition}

\newtheorem{lemma}[theorem]{Lemma}
\newtheorem{assumption}[theorem]{Assumption}
\newtheorem{proposition}[theorem]{Proposition}
\newtheorem{remark}[theorem]{Remark}

\def\field#1{\mathbb #1}%
\def\R{\field{R}}%
\def\N{\field{N}}%

\usepackage{xspace}

\usepackage{enumerate}

\renewcommand{\emptyset}{{\varnothing}}

\title{\LARGE \bf
Verification of Approximate Opacity via Barrier Certificates
}

\author{Siyuan Liu and Majid Zamani
\thanks{This work was supported in part by the H2020 ERC Starting Grant AutoCPS (grant agreement No. 804639), China Scholarship Council, and the NSF under Grant ECCS-2015403.}
\thanks{S. Liu is with the Department of Electrical and Computer Engineering, Technical University of Munich, Germany; email: {\tt\small sy.liu@tum.de}.}
\thanks{M. Zamani is with the Computer Science Department, University of Colorado Boulder, CO 80309, USA. M. Zamani is also with the Computer Science Department, LMU Munich, Germany; email: {\tt\small majid.zamani@colorado.edu}.   
}
}

\begin{document}

\pagestyle{empty}

\maketitle
\thispagestyle{empty}

\begin{abstract}
This paper is motivated by the increasing security concerns of cyber-physical systems. Here, we develop a discretization-free verification scheme targeting an \emph{information-flow} security property, called approximate initial-state opacity, for the class of discrete-time control systems. We propose notions of so-called \emph{augmented control barrier certificates} in conjunction with specified regions of interest capturing the initial and secret sets of the system. Sufficient conditions for (the lack of) approximate initial-state opacity of discrete-time control systems are proposed based on the existence of the proposed barrier certificates. We further present an efficient computation method by casting the conditions for barrier certificates as sum-of-squares programming problems. The effectiveness of the proposed results is illustrated through two numerical examples.  
\end{abstract}

\vspace{-0.3cm}

\section{INTRODUCTION}
	In recent decades, the world has witnessed a rapid increase in the development and deployment of cyber-physical systems (CPSs). Complex CPSs are becoming ubiquitous in safety-critical infrastructure,  ranging from power grids, medical devices to transportation networks. However, the tight interaction between embedded control software and physical processes may release secret information and expose the system to malicious intruders \cite{ashibani2017cyber}. As a result, the security and privacy of modern CPSs have received considerable attentions in the past few years. In this paper, we focus on an \emph{information-flow} security property, called \emph{opacity}, which captures whether a system has the ability to conceal its secret information from outside intruders. 

Opacity has been widely investigated in the domain of discrete event systems (DESs), see \cite{lafortune2018history} and the references therein. Various notions of opacity were proposed depending on how the ``secrets" are defined, e.g., language-based opacity \cite{lin2011opacity}, initial-state opacity \cite{saboori2013verification} and current-state opacity \cite{saboori2007notions}. 
Despite the rich history of opacity in DESs community, most of the works can only apply to systems modeled by finite state automata with discrete state sets or bounded Petri nets \cite{tong2016verification}.
Opacity has been studied on continuous state-space CPSs only very recently \cite{ramasubramanian2019notions, yin2019approximate, liu2020notion}. The results in \cite{ramasubramanian2019notions} introduced a framework for analyzing certain types of initial-state opacity for the class of discrete-time linear systems, wherein the notion of opacity is formulated as an output reachability property. The opacity verification procedure is established by approximating the set of reachable states. 
In \cite{yin2019approximate}, a new notion of \emph{approximate opacity} was proposed that is more applicable to metric systems with continuous state domain. The (possibly) imperfect measurement precision of intruders is characterized by a parameter $\delta$, which on the other hand indicates the security-guarantee level of the system. 
The work was later extended in  \cite{liu2020notion} to the class of discrete-time stochastic systems based on a notion of opacity-preserving simulation functions and their finite abstractions (finite Markov decision processes).    

In this paper, we aim at verifying approximate initial-state opacity property of discrete-time control systems. In this context, the intruder is assumed to be an outside observer that has full knowledge of the system dynamics. The intruder intends to infer if the initial state of the system is one of the secret ones by observing the output trajectories of the system. Approximate initial-state opacity requires that the intruder with measurement precision $\delta$ can never determine that the initial state of the system was a secret one \cite{yin2019approximate}. Verification of initial-state opacity has been widely studied in several literature. Results in DESs literature mostly translate this issue to an initial-state estimation problem and address it by the construction of initial-state estimators \cite{saboori2013verification}. However, this methodology is limited to discrete state-space systems. 
Note that in \cite{yin2019approximate, liu2020notion}, a finite-abstraction based technique was presented for verifying approximate opacity of control systems,
where the original control systems are approximated by discrete ones. Unfortunately, this methodology generally suffers from scalability issues since it requires discretization of the state and input sets of the original system. 

\textls[-2]{Motivated by this limitation, we develop a discretization-free approach for the formal verification of approximate initial-state opacity based on notions of barrier certificates. Barrier certificates have shown to be a promising tool for the analysis of safety problems \cite{prajna2007framework, ames2016control,jagtap2019formal}. A recent attempt to analyze privacy of CPSs using barrier certificates is made in \cite{ahmadi2018privacy}.
A new notion of current-state opacity was considered there 
based on the belief space of the intruder.
The privacy verification problem is cast into checking a safety property of the intruder's belief dynamics using barrier certificates. However, this framework is again limited to systems modeled by partially-observable Markov decision processes (POMDPs) with finite state sets.}

In this work, we first introduce two types of so-called \emph{augmented control barrier certificates} (ACBCs), which are defined for an augmented system constructed by augmenting a control system with itself.
The first type of ACBC guarantees a \emph{safety} property of the augmented system in the sense that there is no trajectory originating from a given initial region reaching a given unsafe set. Along with this, the initial and unsafe regions are designed in a specific form capturing the secret and initial set of the original system. In this way, the existence of an ACBC provides us a sufficient condition ensuring that the original system is approximate initial-state opaque. 
In general, the failure in finding such an ACBC does not mean the system is not opaque. Therefore, we further present another type of ACBC which proves a \emph{reachability} property of the augmented system. This type of ACBC can be utilized for showing that the original system starting from the initial set will eventually reach the unsafe region. This type of ACBC, on the other hand, provides a sufficient condition showing that the original system is lacking approximate initial-state opacity. Additionally, we present a way to compute polynomial ACBCs by means of sum-of-squares (SOS) programming, where the conditions required for the ACBCs are reformulated as SOS constraints.

\textls[-1]{The remainder of the paper is organized as follows. In Section \ref{Sec:II}, we introduce discrete-time control systems and the notion of approximate initial-state opacity.  In Section \ref{SecIII}, we present two types of augmented control barrier certificates and their role in verifying approximate initial-state opacity.
Section \ref{SecIv} presents a methodology for the construction of the augmented control barrier certificates using SOS programming. Section \ref{Sec:ex} illustrates the theoretical results through two examples, and conclusion remarks are given in Section \ref{Sec:conclusion}.} 

\vspace{-0.35cm}

\section{Preliminaries} \label{Sec:II} 
\emph{Notations:} \textls[-10]{We denote by $\R$ and $\N$ the set of real numbers and non-negative integers,  respectively. These symbols are annotated with subscripts to restrict them in
the obvious way, e.g. $\R_{>0}$ denotes the positive real numbers. 
We denote the closed intervals in $\R$ by $[a,b]$. For $a,\!b\!\in\!\N$ and $a\!\le\! b$, we use $[a;b]$ to denote the corresponding intervals in $\N$. 
Given $N \!\!\in \!\!\mathbb N_{\ge 1}$ vectors $x_i \!\in\! \mathbb R^{n_i}$, with $i\!\in\! [1;N]$, $n_i\!\in\! \mathbb N_{\ge 1}$, and $n \!= \!\sum_i n_i$, we denote the concatenated vector in $\mathbb R^{n}$ by $x \!=\! [x_1;\!\ldots\!;x_N]$ and the Euclidean norm of $x$ by $\Vert x\Vert$.
The individual elements in a matrix $A\!\in\! \R^{m\times n}$, are denoted by $\{A\}_{i,j}$, where $i\!\in\![1;m]$ and $j\!\in\![1;n]$. 
We denote by $\emptyset$ the empty set. 
Given sets $X$ and $Y$ with $X\!\subset\! Y$, the complement of $X$ with respect to $Y$ is defined as $Y \!\backslash X \!=\! \{x \!\in\! Y \!\mid\! x \!\notin\! X\}.$
The Cartesian product of two sets $X$ and $Y$ is defined by $X \!\times\! Y\!\!=\!\!\{(x,y) \!\mid\! x \!\in\! X, y \!\in\! Y \}$. 
Given a function $f\!:\! X \!\rightarrow \!Y$ and a function $g\!:\! A \!\rightarrow\! B$, we define $f \!\times\! g \!:\! X \!\times\! A \!\rightarrow\! Y \!\times\! B$. 
For any set $Z\! \subseteq \!\R^{n}$, $\partial Z$ and $\overline Z$, respectively, denotes the boundary and topological closure of $Z$.}
\vspace{-0.5cm}

\subsection{Approximate Initial-State Opacity for Discrete-Time Control Systems}

In this subsection, we introduce the notion of approximate initial-state opacity for 
the class of discrete-time control systems defined below.
\begin{definition}\label{def:sys1}
	A discrete-time control system (dt-CS) $\Sigma$ is defined by the tuple	$\Sigma=(\mathbb X,\mathbb X_0,\mathbb X_s,\mathbb U,f,\mathbb Y,h)$
	where $\mathbb X$, $\mathbb U$, and $\mathbb Y$ are the state set, input set, and output set, respectively. We denote by $\mathbb X_0 \subseteq \mathbb X$ and $\mathbb X_s \subseteq \mathbb X$ the sets of initial states and secret states, respectively. The function $f: \mathbb X\times \mathbb U \rightarrow \mathbb X$ is the state transition function, and $h:\mathbb X \rightarrow \mathbb Y$  is the output function.
	The dt-CS $\Sigma $ is described by difference equations: 
	\begin{align}\label{eq:2}
	\Sigma:\left\{
	\begin{array}{rl}
	\mathbf{x}(t+1)=& f(\mathbf{x}(t),\nu(t)),\\
	\mathbf{y}(t)=&h(\mathbf{x}(t)),
	\end{array}
	\right.
	\end{align}
where $\mathbf{x}:\mathbb{N}\rightarrow \mathbb X $, $\nu: \mathbb{N}\rightarrow \mathbb U$ and $\mathbf{y}:\mathbb{N}\rightarrow \mathbb Y$ are the state, input and output signals, respectively. We use $\mathbf{x}_{x_0,\nu}$ to denote a state run of $\Sigma$ starting from initial state $x_0$ under input run $\nu: \mathbb{N}\rightarrow \mathbb U$.
\end{definition}
	
Throughout this paper, we focus on analyzing approximate initial-state opacity for the class of control systems as in Definition \ref{def:sys1}.
Let us recall the notion of approximate initial-state opacity, originally introduced in \cite{yin2019approximate}, defined below.
\begin{definition}\label{def:opa}
\textls[-1]{Consider a control system $	\Sigma\!=\!(\mathbb X,\mathbb X_0,\mathbb X_s, $ $\mathbb U,f,\mathbb Y,h)$ and a constant $\delta \in \mathbb{R}_{\geq 0}$. System $\Sigma$ is said to be  $\delta$-approximate initial-state opaque if for any $x_0 \!\in\! \mathbb X_0 \!\cap \!\mathbb X_s$ and any finite state run $\mathbf{x}_{x_0,\nu}\!=\!\{x_0,\dots, x_n\}$, there exists $\hat x_0 \!\in\! \mathbb X_0 \!\setminus\! \mathbb X_s$ and a finite state run $\mathbf{x}_{\hat x_0,\hat \nu}\!=\!\{\hat x_0,\dots, \hat x_n\}$ such that} \vspace{-0.25cm}
 \begin{align} \notag
 \max_{i \in [0;n]} \Vert h(x_i)-h(\hat x_i)\Vert \leq \delta.
 \end{align}
\end{definition}
Intuitively, the system $\Sigma$ is $\delta$-approximate initial-state opaque if for every state run initiated from a secret state, there exists another state run originated from a non-secret state with similar output trajectories (captured by $\delta$). 
Hence, the intruder is never certain that the system is initiated from a secret state no matter which output run is generated. 
Hereafter, we assume without loss of generality that $\forall x_0 \!\in\! \mathbb X_0 \!\cap\! \mathbb X_s$,   \vspace{-0.15cm}
\begin{align} \label{initassum}
\{x \in \mathbb X_0 \mid \Vert h(x) - h(x_0)\Vert \leq \delta \} \nsubseteq \mathbb X_s.
\end{align} 
This assumption requires that the secret of the system is not revealed initially; otherwise approximate initial-state opacity is trivially violated.
\begin{remark}
We remark that our notion of initial-state opacity is different from that of \emph{observability}. 
		An observability notion states that every initial state can be determined by observing a finite output sequence under a given input run \cite{hermann1977nonlinear}. However, in our context, initial-state opacity is defined as the plausible deniability of a system for every secret initial information under any input sequence.
	In DESs literature, it was shown that observability can be reformulated as language-based opacity by properly specifying the languages and the observation mapping \cite{lin2011opacity}. However, the relationship between opacity and observability is more challenging in the domain of CPSs and is left to future investigation.
\end{remark}
\vspace{-0.2cm}
\section{Verifying Approximate Initial-State Opacity} \label{SecIII}
Although the verification of opacity has been widely investigated for finite systems with discrete states (e.g., in the domain of DES), there is no systematic way in the literature to check opacity of systems with continuous state spaces. 
In the sequel, we propose a technique that is sound in verifying approximate initial-state opacity for discrete-time control systems.  
Our approach is based on finding a certain type of barrier certificates as defined in the next subsection. 
\vspace{-0.2cm}
\subsection{Verifying Approximate Initial-State Opacity via Barrier Certificates}
Consider a dt-CS $\Sigma\!=\!(\mathbb X,\mathbb X_0,\mathbb X_s,\mathbb U,f,\mathbb Y,h)$ as in Definition \ref{def:sys1}.  
We define the associated augmented system by 
\begin{align} \notag
\Sigma \!\times\! \Sigma\!=\!(\mathbb X \!\times\! \mathbb X,\mathbb X_0 \!\times\! \mathbb X_0,\mathbb X_s \!\times\! \mathbb X_s,\mathbb U \!\times\! \mathbb U,f \!\times\!f,\mathbb Y \!\times\! \mathbb Y, h \!\times\! h),
\end{align}
which can be seen as the product of a dt-CS $\Sigma$ and itself. 
For later use, we denote by $(x,\hat x) \!\in\! \mathbb X \!\times\! \mathbb X$ a pair of states in $\Sigma \!\times \!\Sigma$ and by $(\mathbf{x}_{x_0,\nu},  \mathbf{x}_{\hat x_0,\hat \nu})$ the state trajectory of $\Sigma \times \Sigma$ starting from $(x_0, \hat x_0)$ under input run ($\nu$, $\hat \nu$). We use  $\mathcal{R}\!=\!\mathbb X \!\times \mathbb X$ to denote the augmented state space.

Now, we define a notion of barrier certificates that is constructed over the augmented system $\Sigma \!\times \!\Sigma$ and ensures a safety property for $\Sigma \!\times\! \Sigma$.
\begin{proposition} \label{BC}
	Consider a dt-CS $\Sigma$ as in Definition \ref{def:sys1}, the associated augmented system $\Sigma \!\times\! \Sigma$, and sets $\mathcal{R}_0, \mathcal{R}_{u} \!\subseteq\! \mathcal{R}$. 
	Suppose there exists a function $\mathcal{B}:\mathbb X \times \mathbb X \!\rightarrow\! \mathbb{R}$ and constants $\underline{\epsilon},\overline{\epsilon} \in \mathbb{R}$ with $\overline{\epsilon} > \underline{\epsilon}$ such that \vspace{-0.2cm}
	\begin{align} \label{first}
	&\forall (x,\hat x) \in \mathcal{R}_0, \quad \quad \quad \quad \quad \quad  \quad \quad   \quad \quad \mathcal{B}(x,\hat x) \leq \underline{\epsilon}, \\ \label{second}
	&\forall (x,\hat x) \in \mathcal{R}_u, \quad \quad \quad \quad \quad \quad  \quad \quad  \quad \quad  \mathcal{B}(x,\hat x) \geq \overline{\epsilon}, \\  \vspace{-0.2cm} \notag
	&\forall (x,\hat x) \in \mathcal{R}, \forall u \in \mathbb{U}, \exists \hat u \in \mathbb{U}, \\ \label{third} \vspace{-0.2cm}
	& \quad \quad \quad \quad \quad \quad \quad \mathcal{B}(f(x,u),f(\hat x,\hat u)) - \mathcal{B}(x,\hat x)\leq 0. 
	\end{align}\\[-14pt]
Then, for any initial condition  $(x_0, \hat x_0) \!\in\! \mathcal{R}_0$ and for any input run $\nu$, there exists an input run $\hat \nu$ such that
$(\mathbf{x}_{x_0,\nu}(t), \mathbf{x}_{\hat x_0,\hat \nu}(t)) \cap \mathcal{R}_u = \emptyset$, $\forall t \in \mathbb{N}$.

\end{proposition}

\begin{proof}
	This proposition is proved by contradiction. Consider a state trajectory $(\mathbf{x}_{x_0,\nu},\! \mathbf{x}_{\hat x_0,\hat \nu})$ of $\Sigma \!\times\! \Sigma$ that starts from an initial condition $(x_0, \hat x_0) \!\in\! \mathcal{R}_0$, under input sequences $\nu$ and $\hat \nu$. Suppose $\hat \nu$ is computed such that the inequality in \eqref{third} holds. Assume the state run reaches a state in $\mathcal{R}_u$, i.e., $(\mathbf{x}_{x_0,\nu}(t), \mathbf{x}_{\hat x_0,\hat \nu}(t)) \!\in \! \mathcal{R}_u$ for some $t \in \mathbb{N}$. From \eqref{first} and \eqref{second}, we have  $\mathcal{B}(x_0,\hat x_0)  \!\leq \! \underline{\epsilon}$ and $\mathcal{B}(\mathbf{x}_{x_0,\nu}(t), \mathbf{x}_{\hat x_0,\hat \nu}(t)) \! \geq \! \overline{\epsilon}$. By using \eqref{third}, one has $\overline{\epsilon}  \! \leq  \! \mathcal{B}(\mathbf{x}_{x_0,\nu}(t), \mathbf{x}_{\hat x_0,\hat \nu}(t))  \!\leq \! \mathcal{B}(x_0,\hat x_0)  \!\leq  \!\underline{\epsilon}$, which contradicts $\overline{\epsilon}  \!> \! \underline{\epsilon}$.
	Therefore, for any state trajectory of $\Sigma \!\times\! \Sigma$ starting from any initial condition in $\mathcal{R}_0$ under any input run $\nu$, $(\mathbf{x}_{x_0,\nu}(t), \mathbf{x}_{\hat x_0,\hat \nu}(t)) \cap \mathcal{R}_u\! =\! \emptyset$ always holds under the extracted control policy $\hat \nu$, which completes the proof. 
\end{proof}

If $\mathcal{B}(x,\hat x)$ satisfies the conditions in Proposition \ref{BC}, then it is called an \emph{augmented control barrier certificate} (ACBC) for $\Sigma \!\times \! \Sigma$.  
Next, we show how one can leverage the ACBC to verify approximate initial-state opacity for a dt-CS $\Sigma$. To this purpose, we define the sets of initial conditions $\mathcal{R}_0$ and unsafe states $\mathcal{R}_u$ as: \vspace{-0.1cm}
\begin{align} \label{set:initial}
\!\!\!\!\!\mathcal{R}_0\!=&\{(x,\hat x) \!\in \!(\mathbb X_0 \!\cap\! \mathbb X_s) \!\times\! (\mathbb X_0 \!\setminus \!\mathbb X_s) \!\mid \Vert h(x)\!-\!h(\hat x)\Vert \!\leq\! \delta\},\!\!\! \\ \label{set:unsafe}
\!\!\!\!\!\mathcal{R}_u\!=&\{(x,\hat x)\!\in \!\mathbb X \times \mathbb X \mid \Vert h(x)-h(\hat x)\Vert > \delta \},
\end{align}
where $\delta \in \mathbb{R}_{\geq 0}$ captures the measurement precision of the intruder as introduced in Definition \ref{def:opa}. The following theorem provides us a sufficient condition in verifying approximate initial-state opacity of discrete-time control systems.
\begin{theorem}
	Consider a dt-CS $\Sigma$ as in Definition \ref{def:sys1}. Suppose there exists a function $\mathcal{B}:\mathbb X \times \mathbb X \rightarrow \mathbb{R}$ satisfying \eqref{first}-\eqref{third} in Proposition \ref{BC} with sets $\mathcal{R}_0, \mathcal{R}_u$ given in \eqref{set:initial}-\eqref{set:unsafe}. Then, system $\Sigma$ is $\delta$-approximate initial-state opaque. 
\end{theorem}
\begin{proof}
Consider an arbitrary secret initial state $x_0 \!\in\! \mathbb X_0 \cap \mathbb X_s$, any input run $\nu$, and the corresponding state run $\mathbf{x}_{x_0,\nu}$ in $\Sigma$. 
First note that by \eqref{initassum}, $\{x \!\in\! \mathbb X_0 \mid \Vert h(x) \!-\! h(x_0)\Vert \!\leq\! \delta \} \!\nsubseteq\! \mathbb X_s$. It follows that there exists an initial state $\hat x_0 \!\in\! \mathbb X_0 \!\setminus\! \mathbb X_s$ such that $\Vert h(\hat x_0) \!-\! h(x_0) \Vert \!\leq\! \delta$.
Consider the pair of initial states $(x_0,\hat x_0)$. It can be readily seen that $(x_0,\hat x_0) \in \mathcal{R}_0$ as in \eqref{set:initial}. 
Now, given the existence of an ACBC as in Proposition \ref{BC}, there exists a control policy $\hat \nu$ such that \eqref{third} is satisfied. By using Proposition \ref{BC}, under $\hat \nu$, we have the guarantee that any state run of $\Sigma \times \Sigma$ starting from $\mathcal{R}_0$ never reaches the unsafe region $\mathcal{R}_u$, i.e. $(\mathbf{x}_{x_0,\nu}(t), \mathbf{x}_{\hat x_0,\hat \nu}(t)) \!\cap\! \mathcal{R}_u \!=\! \emptyset$, $\forall t \!\in\! \mathbb{N}$.  This simply implies the satisfaction of $\Vert h(\mathbf{x}_{x_0,\nu}(t))\!-\!h(\mathbf{x}_{\hat x_0,\hat \nu}(t))\Vert \!\leq\! \delta$, $\forall t \!\in\! \mathbb{N}$.
Since $x_0 \!\in\! \mathbb X_0 \!\cap\! \mathbb X_s$ and $\mathbf{x}_{x_0,\nu}$ are arbitrarily chosen, 
we conclude that $\Sigma$ is $\delta$-approximate initial-state opaque.
\end{proof}
\vspace{-0.25cm}
\subsection{Verifying Lack of Approximate Initial-State Opacity via Barrier Certificates}
We presented in the previous subsection a sufficient condition for verifying approximate initial-state opacity. However, failing to find such an ACBC does not necessarily imply that the system is not opaque. Motivated by this, in this subsection, we aim at presenting a sufficient condition to verify the lack of approximate initial-state opacity of a dt-CS $\Sigma$. This method is based on constructing another type of ACBC ensuring a reachability property for $\Sigma \!\times\! \Sigma$. 
\begin{proposition} \label{BC1}
	Consider a dt-CS $\Sigma$ as in Definition \ref{def:sys1}, the associated augmented system $\Sigma \times \Sigma$, and sets $\mathcal{R}_0, \mathcal{R}_{u} \subseteq \mathcal{R}$. 
Suppose $\mathbb X \subset \mathbb{R}^n$ is a bounded set and there exists a continuous function $V:\mathbb X \times \mathbb X \rightarrow \mathbb{R}$ such that \vspace{-0.2cm}
	\begin{align} \label{1st}
	&\forall (x,\hat x) \in \mathcal{R}_0, \quad \quad \quad \quad \quad \quad  \quad \quad   \quad \quad V(x,\hat x) \leq 0, \\ \label{2nd}
	&\forall(x,\hat x) \in \partial \mathcal{R} \setminus \partial  \mathcal{R}_u, \quad \quad \quad \quad \quad \quad  \quad  V(x,\hat x) > 0, \\  \notag
	&\forall(x,\hat x) \in \overline{(\mathcal{R} \setminus \mathcal{R}_u)},  \exists u \in \mathbb{U}, \forall \hat u \in \mathbb{U},\\ \label{3rd}
	& \quad \quad \quad \quad \quad \quad \quad V(f(x,u),f(\hat x,\hat u))-V(x,\hat x) < 0.
	\end{align}\\[-14pt]
Then, for any initial condition $(x_0, \hat x_0) \in \mathcal{R}_0$, there exists an input run $\nu$ such that $(\mathbf{x}_{x_0,\nu}(T), \mathbf{x}_{\hat x_0,\hat \nu}(T)) \in \mathcal{R}_u$ for any $\hat \nu$, for some $T \geq 0$, and $(\mathbf{x}_{x_0,\nu}(t), \mathbf{x}_{\hat x_0,\hat \nu}(t)) \in \mathcal{R}$, $\forall t \in [0,T]$. 
\end{proposition}
\begin{proof}
\textls[-12]{Consider an initial state $(x_0, \hat x_0)\!\!\in\!\! \mathcal{R}_0$. One has $V(x_0, \hat x_0) \!\!\leq\!\! 0$ by \eqref{1st}. 
Consider an input run $\nu$ such that \eqref{3rd} is satisfied for
the state runs $\mathbf{x}_{x_0,\nu}(t), \mathbf{x}_{\hat x_0,\hat \nu}(t)$ of $\Sigma$, where $\hat \nu$ is an arbitrary input run. 
First note that the continuous function $V(x,\hat x)$ is bounded below on the compact set $\overline{(\mathcal{R} \!\setminus\! \mathcal{R}_u)}$. 
From \eqref{3rd}, $V(x,\hat x)$ is strictly decreasing along the trajectory $(\mathbf{x}_{x_0,\nu},\! \mathbf{x}_{\hat x_0,\hat \nu})$ in region $\overline{(\mathcal{R} \!\setminus \!\mathcal{R}_u)}$.   
It follows that $(\mathbf{x}_{x_0,\nu},\! \mathbf{x}_{\hat x_0,\hat \nu})$ must leave $\overline{(\mathcal{R} \!\setminus\! \mathcal{R}_u)}$ in finite time.
Now, assume $(\mathbf{x}_{x_0,\nu}, \mathbf{x}_{\hat x_0,\hat \nu})$ leaves $\overline{(\mathcal{R} \!\setminus\! \mathcal{R}_u)}$ without entering region $\mathcal{R}_u$ first. Consider the first time instant $t\! =\! T$ when $(\mathbf{x}_{x_0,\nu}(t), \mathbf{x}_{\hat x_0,\hat \nu}(t))$ is leaving $\overline{(\mathcal{R} \!\setminus\! \mathcal{R}_u)}$, i.e. $(\mathbf{x}_{x_0,\nu}(t),$ $ \mathbf{x}_{\hat x_0,\hat \nu}(t)) \!\in \!\overline{(\mathcal{R} \!\setminus\! \mathcal{R}_u)}$ for all $t \!\in\! [0,T]$, and $(\mathbf{x}_{x_0,\nu}(T\!+\!\epsilon), \mathbf{x}_{\hat x_0,\hat \nu}(T\!+\!\epsilon))\! \notin \!\mathcal{R}$ for any $\epsilon \!\!>\!\!0$. By \eqref{3rd} and $V(x_0, \hat x_0) \!\leq\! 0$, we have $V(\mathbf{x}_{x_0,\nu}(T), \mathbf{x}_{\hat x_0,\hat \nu}(T)) \!\leq \!0$ which contradicts \eqref{2nd}. Therefore, we conclude that for any run starting from $\mathcal{R}_0$ under $\nu$, there must exist $T \!\!\geq\!\! 0$ such that $(\mathbf{x}_{x_0,\nu}(T), \mathbf{x}_{\hat x_0,\hat \nu}(T)) \!\in\! \mathcal{R}_u$ for any $\hat \nu$, and $(\mathbf{x}_{x_0,\nu}(t), \mathbf{x}_{\hat x_0,\hat \nu}(t)) \!\in\! \mathcal{R}$, $\forall t \!\in\! [0,T]$, which completes the proof.}
\end{proof}
\begin{remark}
We remark that the universal quantifier in \eqref{1st} is not necessary to show the lack of approximate initial-state opacity. 
One can relax the universal quantifier to an existential one by modifying the definition of barrier certificates, together with the corresponding initial and unsafe regions, at the cost of having a much more complex structure.  
However, it is difficult to formulate such a function and the corresponding set constraints to sum-of-squares programs (c.f. Section \ref{SecIv}), and thus, is out of the scope of this paper.

\end{remark}

\textls[-3]{A function $V(x,\hat x)$ satisfying the conditions in Proposition \ref{BC1} is also called an ACBC for $\Sigma \!\times\! \Sigma$. The idea of using barrier functions to prove reachability was first described in \cite{prajna2005primal}. 
Next, we show that the above-defined ACBC can be used for verifying the lack of approximate initial-state opacity of dt-CSs.}

\begin{theorem}
\textls[-1]{Consider a dt-CS $\Sigma$ as in Definition \ref{def:sys1}. Suppose there exists a continuous function $V\!:\!\mathbb X \!\times\! \mathbb X \!\rightarrow\! \mathbb{R}$ satisfying \eqref{1st}-\eqref{3rd} in Proposition \ref{BC1}  with sets $\mathcal{R}_0, \mathcal{R}_u$ given in \eqref{set:initial}-\eqref{set:unsafe}. Then, system $\Sigma$ is not $\delta$-approximate initial-state opaque. }	
\end{theorem}
\begin{proof}
\textls[-1]{First note that from Definition \ref{def:opa}, system $\Sigma$ is not $\delta$-approximate initial-state opaque if there exists a state run $\mathbf{x}_{x_0,\nu}$ with $x_0 \!\in\! \mathbb X_0 \!\cap\! \mathbb X_s$, such that for any other state runs $\mathbf{x}_{\hat x_0,\hat \nu}$ starting from a non-secret initial condition $\hat x_0 \!\in\! \mathbb X_0 \!\setminus\! \mathbb X_s$, $\max_{i \in [0;n]} \Vert h(x_i)\!-\!h(\hat x_i)\Vert \!>\! \delta$ holds. 
	Now consider a function $V:\mathbb X \!\times\! \mathbb X \!\rightarrow\! \mathbb{R}$ and an input run $\nu$ satisfying \eqref{3rd}. 
	Then, by Proposition \ref{BC1} 
	and from \eqref{set:initial}-\eqref{set:unsafe}, it follows that there must exist a secret state $x_0 \in \mathbb X_0 \cap \mathbb X_s$ and a state run $\mathbf{x}_{x_0,\nu}$ under input run $\nu$, such that for any trajectory $\mathbf{x}_{\hat x_0,\hat \nu}$ originated from any non-secret initial condition $\hat x_0 \!\in\! \mathbb X_0 \!\setminus\! \mathbb X_s$, the trajectories $(\mathbf{x}_{x_0,\nu}(t),\! \mathbf{x}_{\hat x_0,\hat \nu}(t))$ will eventually reach $\mathcal{R}_u$ in finite time, where $\Vert h(\mathbf{x}_{x_0,\nu}(t))\!-\!h(\mathbf{x}_{\hat x_0,\hat \nu}(t))\Vert \!>\! \delta$. 
	Therefore, for the state run $\mathbf{x}_{x_0,\nu}(t)$, there does not exist a state run starting from a non-secret initial state that generates similar output trajectories. Thus, $\delta$-approximate initial-state opacity is violated. }
\end{proof}

In the next section, we discuss how to leverage existing computational methods and software tools to compute $\mathcal{B}(x,\hat x)$ and $V(x,\hat x)$ in Propositions \ref{BC} and \ref{BC1}, respectively. 
\vspace{-0.2cm}
\section{Computation of Barrier Certificates using Sum-of-Squares Technique} \label{SecIv}
In the previous section, we presented sufficient conditions for verifying (resp. the lack of) approximate initial-state opacity of discrete-time control systems by searching for  barrier certificates satisfying inequalities (resp. \eqref{1st}-\eqref{3rd}) \eqref{first}-\eqref{third}.
For systems with polynomial transition functions and semi-algebraic  sets (i.e., described by polynomial equalities and inequalities) $\mathbb{X}_0$, $\mathbb{X}_s$, and $\mathbb{X}$, 
an efficient computational method based on sum-of-squares (SOS) programming can be utilized to search for polynomial barrier certificates.

\begin{assumption}\label{assump1}
	System $\Sigma$ has continuous state set $\mathbb X \subseteq \mathbb{R}^n$ and continuous input set $\mathbb U \subseteq \mathbb{R}^m$. Its transition function $f: \mathbb X\times \mathbb U\rightarrow \mathbb X$ is polynomial in variables $x$ and $u$, and output map $h$ is polynomial in variable $x$.
\end{assumption}	

In the next lemma, we translate the conditions in Proposition \ref{BC} to SOS constraints.
\begin{lemma} \label{SOS}
	\textls[-6]{Suppose Assumption \ref{assump1} holds and sets $\mathcal{R}_0$, $\mathcal{R}_u$, $\mathcal R$, and $\mathbb U$ can be defined as $\mathcal{R}_0\!=\! \{(x,\hat x)\!\in\! \mathbb{R}^n \!\times\! \mathbb{R}^n \mid g_0(x,\hat x)\!\geq\! 0\}$, $\mathcal{R}_u\! =\! \{(x,\hat x)\!\in\! \mathbb{R}^n \!\times\! \mathbb{R}^n \mid g_u(x,\hat x)\!\geq\! 0\}$, $\mathcal R\!= \!\{(x,\hat x) \!\in\! \mathbb{R}^n \!\times \!\mathbb{R}^n \mid g(x,\hat x)\!\geq \!0\}$, $\mathbb U\!= \!\{u \!\in\! \mathbb{R}^m \mid g_c(u)\!\geq \!0\}$, where the inequalities are defined element-wise, and $g_0,g_u,g,g_c$ are vectors of some polynomial functions. Suppose there exists a polynomial function $\mathcal{B}(x,\hat x)$, polynomials $p_{\hat u_i}(x,\hat x,u)$ corresponding to the $i^{th}$ component of $\hat u\!=\![\hat u_1;\ldots;\hat u_m] \!\in\!\mathbb U \!\subseteq\! \mathbb{R}^m$, and vectors of SOS polynomials $\lambda_0,\lambda_u,\lambda,\lambda_c$ of appropriate size such that the following expressions are SOS polynomials:} \vspace{-0.2cm}
	\begin{align} \label{SOSpoly1}
	-&\mathcal{B}(x,\hat x)-\lambda_0^{\top}(x,\hat x)g_0(x,\hat x)+\underline{\epsilon},\\
	&\mathcal{B}(x,\hat x)-\lambda_u^{\top}(x,\hat x)g_u(x,\hat x)-\overline{\epsilon},\\ \notag 
	-&\mathcal{B}(f(x,u),f(\hat x,\hat u)) + \mathcal{B}(x,\hat x) - \lambda^{\top}(x,\hat x)g(x,\hat x) \\   \label{SOSpoly3}    
	&\quad \quad \quad -\sum_{i=1}^{m}(\hat u_i-p_{\hat u_i}(x,\hat x,u))- \lambda^{\top}_c(u)g_c(u),  
	\end{align}\\[-10pt]
\textls[-1]{where  $\underline{\epsilon},\overline{\epsilon} \!\in\! \mathbb{R}_{\geq 0}$ are some constants with $\overline{\epsilon} > \underline{\epsilon}$. Then, $\mathcal{B}(x,\hat x)$ satisfies conditions \eqref{first}-\eqref{third} and
	$\hat u\!=\![\hat u_1;\hat u_2;\ldots;\hat u_m]$, where $\hat u_i\!=\!p_{\hat u_i}(x,\hat x,u)$, $\forall i \!\in\! [1;m]$, is a control policy satisfying $\eqref{third}$. }
\end{lemma}

We omit the proof of Lemma \ref{SOS}, since it follows the general methods for converting set constraints conditions to SOS programs with \emph{Positivstellensatz} conditions, see \cite{SOSTOOLS} for details.
Similarly, we convert the conditions of Proposition \ref{BC1} to SOS constraints as well.

\begin{lemma} \label{SOS1}
\textls[-6]{Suppose Assumption \ref{assump1} holds and $\mathbb X \subset \mathbb{R}^n$ is a bounded set. 
  Suppose the regions of interest in Proposition \ref{BC1}
	can be defined as $\mathcal{R}_0 \!= \!\{(x,\hat x)\!\in\! \mathbb{R}^n \!\times\! \mathbb{R}^n \! \mid\!  g_0(x,\hat x)\!\geq 0\}$, 
	$\partial \mathcal{R}  \!\setminus \! \partial  \mathcal{R}_u \! =\! \{(x,\hat x)\!\in\! \mathbb{R}^n \!\times\! \mathbb{R}^n \mid g_u(x,\hat x)\!\geq\! 0\}$, $\overline{(\mathcal{R}  \!\setminus \! \mathcal{R}_u)}\!= \!\{(x,\hat x) \!\in\! \mathbb{R}^n \!\times \!\mathbb{R}^n \mid g(x,\hat x)\!\geq \!0\}$,	
	$\mathbb U\!= \!\{\hat u \!\in\! \mathbb{R}^m \!\mid\! g_c(\hat u)\!\geq \!0\}$, where the inequalities are defined element-wise, and $g_0,g_u,g,g_c$ are vectors of some polynomial functions. Suppose there exists a polynomial function $V(x,\hat x)$, polynomials $p_{u_i}(x,\hat x,\hat u)$ corresponding to the $i^{th}$ component of $u\!=\![u_1;\ldots;u_m] \!\in\!\mathbb U \!\subseteq\! \mathbb{R}^m$, 
	and vectors of SOS polynomials $\lambda_0,\lambda_u,\lambda,\lambda_c$ of appropriate size such that the following expressions are SOS polynomials:} \vspace{-0.2cm}
	\begin{align} \label{SOSpoly4}
	-&V(x,\hat x)-\lambda_0^{\top}(x,\hat x)g_0(x,\hat x),\\ \label{ineq2}
	&V(x,\hat x)-\lambda_u^{\top}(x,\hat x)g_u(x,\hat x)-\varepsilon,\\  \notag
	-&V(f(x,u),f(\hat x,\hat u)) + V(x,\hat x) - \lambda^{\top}(x,\hat x)g(x,\hat x)\\ \label{SOSpoly7}
	& \quad \quad -\sum_{i=1}^{m}(u_i-p_{u_i}(x,\hat x,\hat u))- \lambda^{\top}_c(\hat u)g_c(\hat u)-\varepsilon,
	\end{align}\\[-10pt]
	where $\varepsilon$ is a small positive number. Then, $V(x,\hat x)$ satisfies conditions \eqref{1st}-\eqref{3rd} and $u\!=\![u_1; u_2;\ldots;u_m]$, where $u_i\!=\!p_{u_i}(x,\hat x,\hat u)$, $\forall i \in [1;m]$, is a control policy satisfying $\eqref{3rd}$. 
\end{lemma}
Note that a small tolerance $\varepsilon$ in \eqref{ineq2} and \eqref{SOSpoly7} is needed to ensure positivity of polynomials as required in \eqref{2nd} and \eqref{3rd}.

\begin{remark}
As seen in Lemmas \ref{SOS} and \ref{SOS1}, in order to search for polynomial barrier certificates by means of SOS programming, it is required that regions $\mathcal{R}_0$, $\mathcal{R}_u$, $\partial \mathcal{R} \setminus \partial  \mathcal{R}_u$, $\overline{(\mathcal{R}  \!\setminus \! \mathcal{R}_u)}$ are semi-algebraic sets. 
We highlight that having a system $\Sigma$ with semi-algebraic sets $\mathbb{X}_0$, $\mathbb{X}_s$, and $\mathbb{X}$ is enough to ensure that all these regions are semi-algebraic.
In particular, as a consequence of Tarski-Seidenberg principle \cite{tarski1998decision}, the class of all semi-algebraic sets is closed under finite unions, intersections, taking complement, and Cartesian product. The boundary, the interior, and the closure of a semi-algebraic set are also semi-algebraic. 
Additionally, given the polynomial output map $h$, the set of states satisfying $\Vert h(x)\!-\!h(\hat x)\Vert \!\leq\! \delta$ is equivalent to the one satisfying $(h(x)\!-\!h(\hat x))^{\top}(h(x)\!-\!h(\hat x)) \!\leq\! \delta^2$,
which is again a semi-algebraic set.
See \cite{coste2000introduction} for details. 
\end{remark}

One can leverage existing computational toolboxes such as SOSTOOLS \cite{SOSTOOLS} together with semidefinite programming solvers such as SeDuMi \cite{sturm1999using} to compute polynomial barrier certificates satisfying \eqref{SOSpoly1}-\eqref{SOSpoly3} or \eqref{SOSpoly4}-\eqref{SOSpoly7}.

\begin{remark}
\textls[-3]{By formulating conditions (3)-(5) (resp. (8)-(10)) as a satisfiability problem, one can alternatively search for parametric control barrier certificates using an iterative program synthesis framework, called Counter-Example-Guided Inductive Synthesis (CEGIS), with the help of Satisfiability Modulo Theories (SMT) solvers such as Z3 \cite{de2008z3} and dReal \cite{gao2013dreal}; see, e.g., \cite{jagtap2019formal} for more details. We also refer interested readers to the recent work \cite{peruffo2020automated}, where machine learning techniques were exploited for the construction of barrier certificates.}
\end{remark}

\vspace{-0.2cm}
\section{Examples} \label{Sec:ex}
In this section, we provide two examples to illustrate how 
one can utilize the theoretical results obtained in Sections \ref{SecIII} and \ref{SecIv} for the verification of (the lack of) approximate initial-state opacity. 
\vspace{-0.2cm}
\subsection{Verifying Approximate Initial-State Opacity on a Vehicle Model}

In this example, we consider an autonomous vehicle moving on a single lane road, whose state variable is defined as $x \!=\! [x_1;x_2]$, with $x_1$ being its absolute position (in the road frame) and $x_2$ being its absolute velocity. The discrete-time dynamics of the vehicle is modeled as: \vspace{-0.1cm}
\begin{align} \notag
	\begin{bmatrix}
		x_1(t+1)\\x_2(t+1)
	\end{bmatrix}	
	&= \begin{bmatrix}
		1& \Delta \tau\\
		0& 1
	\end{bmatrix}\begin{bmatrix}
		x_1(t)\\x_2(t)
	\end{bmatrix}+\begin{bmatrix}
		\Delta \tau^2/2\\ \Delta \tau
	\end{bmatrix}u(t),\\   \label{car}  
	y(t) &= \begin{bmatrix}
		1& 0
	\end{bmatrix}\begin{bmatrix}
		x_1(t)\\x_2(t)
	\end{bmatrix},
\end{align} \\[-11pt]
\textls[-6]{where $u$ is the control input (acceleration) and $\Delta \tau$ is the sampling time. The output is assumed to be the position of the vehicle on the road.
Let us first briefly explain the motivation behind this example; see Fig.~\ref{fig1}.}
\begin{figure}[tb!]
		\vspace{0.15cm}
	\centerline{
	\includegraphics[width=.23\textwidth]{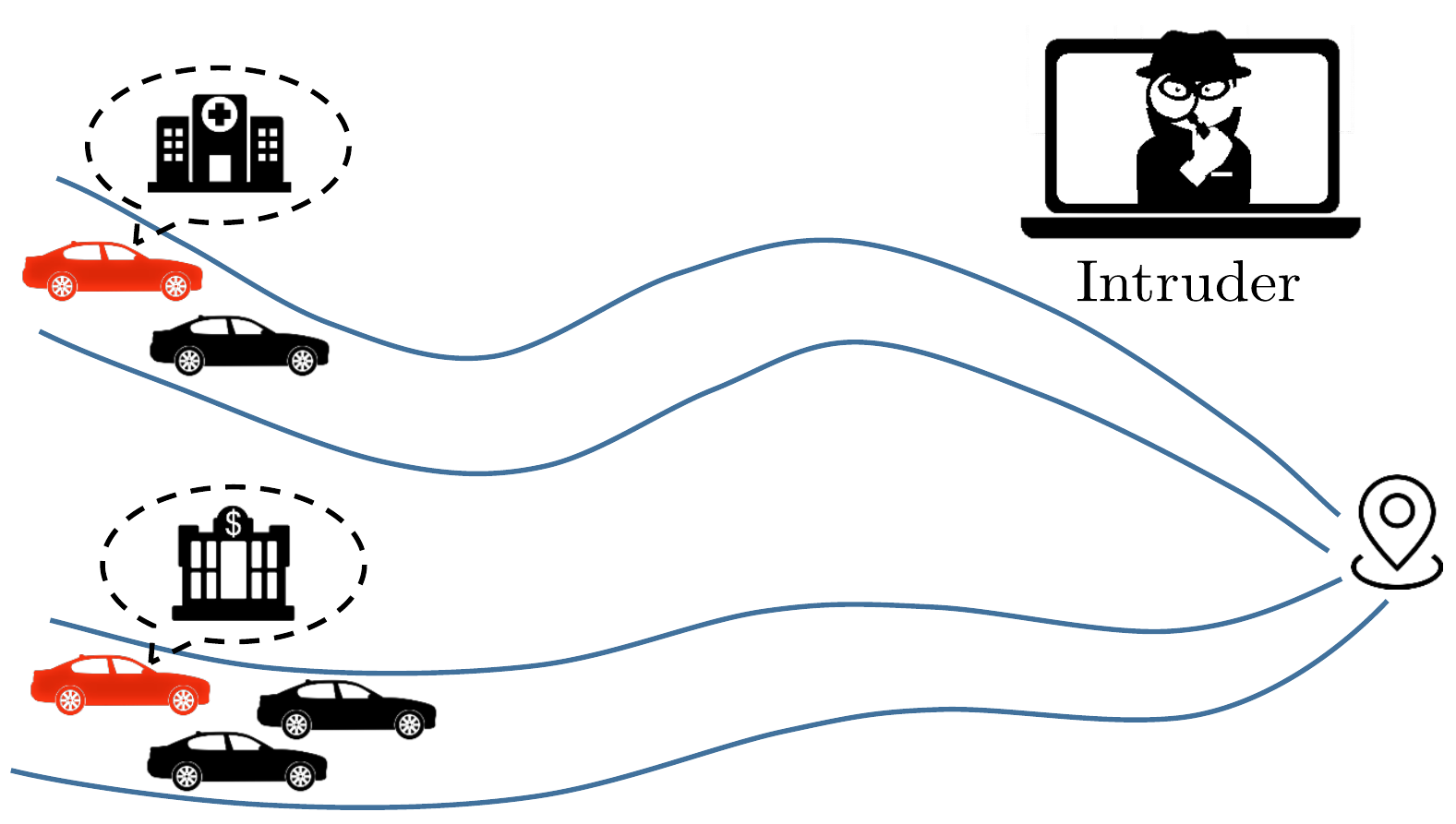}
}
	\caption{Plausible deniability of a vehicle in terms of its initial conditions. The blue lines roughly indicate the intruder's insufficient observation precision.}
	\label{fig1}
		\vspace{-0.65cm}
\end{figure}
\textls[-6]{Suppose the initial locations of the vehicle contain critical information which is needed to be kept secret, e.g., the vehicle might be a cash transit van that aims at transferring money initially from a bank to an ATM machine, or a patient who initially visited a hospital but unwilling to reveal personal information to others. It is implicitly assumed that there is a malicious intruder who is observing the behavior of the vehicle remotely intending to carry out an attack.
Therefore, it is in the interest of the system to verify whether it maintains plausible deniability for secret initial conditions where some confidential assignment is executed. This problem can be formulated as a $\delta$-approximate initial-state opacity problem, where $\delta \geq 0$ captures the security-guarantee level in terms of the measurement precision of the intruder.  
Now consider system \eqref{car} with state space $\mathbb{X} \!=\! [0, 10] \!\times\! [0, 0.1]$, initial set $\mathbb{X}_0\!=\! [0, 10] \!\times\! \{0\}$, secret set $\mathbb{X}_s \!=\! [0, 1] \!\times\! [0, 0.1]$, input set $\mathbb{U} \!=\! [-0.05, 0.05]$ and sampling time $\Delta \tau \!=\! 1$. 
Consider the augmented system $\Sigma \times \Sigma$. Accordingly, the regions of interest in \eqref{set:initial} and \eqref{set:unsafe} are 
$\mathcal{R}_0 \!=\! \{[x_1;x_2] \!\in\! [0,1] \!\times\! \{0\},[\hat x_1;\hat x_2] \!\in\! [1,10] \!\times\! \{0\} \mid (x_1 \!- \!\hat x_1)^2 \!\leq\! \delta^2\}$, and $\mathcal{R}_u \!=\! \{(x, \hat x) \!\in\! \mathbb{X} \!\times\! \mathbb{X} \mid (x_1 \!-\! \hat x_1)^2 \!\geq\! \delta^2 \!+\! \epsilon \}$.
Note that a small positive number $\epsilon$ is needed to certify positivity of the obtained polynomials using SOS programming. 
Now, we set the threshold parameter to be $\delta \!=\! 1$ and search for barrier certificates 
by solving sum-of-squares programs with the help of SOSTOOLS and SeDuMi tools as described in Section \ref{SecIv}. Using Lemma \ref{SOS},  we obtained a polynomial ACBC of degree 2 satisfying \eqref{SOSpoly1}-\eqref{SOSpoly3} with $\underline{\epsilon} = 1$, $\overline{\epsilon} = 1.001$ and a tolerance $\epsilon=0.01$ as follows} \vspace{-0.16cm}
\begin{align} \notag
&\mathcal{B}(x,\hat x)= 0.9227x_1^2+0.2348x_2^2+0.9227\hat x_1^2+0.2348\hat x_2^2\\\notag
&+0.006x_1x_2-0.006\hat x_1x_2-0.006x_1\hat x_2-0.006\hat x_1\hat x_2\\ \vspace{-0.15cm} \notag
&-0.4696x_2\hat x_2-1.845x_1\hat x_1-0.0002\hat x_1+0.0728,
\end{align}  \\[-15pt]
\textls[-6]{and the corresponding control policy is 
$\hat u(x,\hat x,u)= 0.8x_1- 0.8x_2 + 1.5 \hat x_1 -1.5\hat x_2 + u$.
Therefore, we conclude that $\Sigma$ is $1$-approximate initial-state opaque. Particularly, for every trajectory starting from a secret state, there always exists at least one alternative trajectory originated from a non-secret state which are indistinguishable for an intruder with measurement precision $\delta$.}
\begin{figure}[tb!]
		\vspace{0.15cm}
	\centering
	\begin{overpic}[width=.21\textwidth]{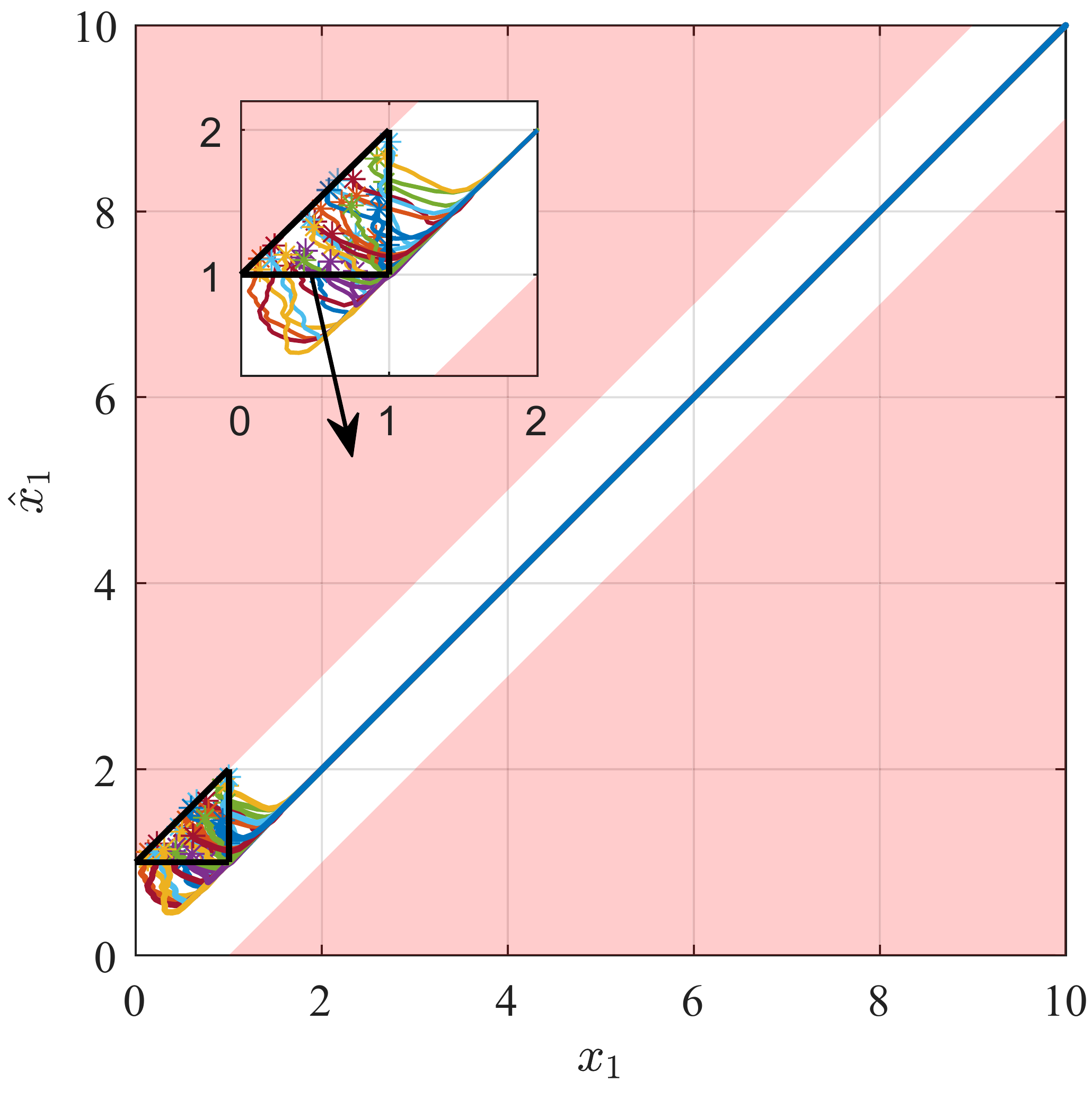}
		\put(66.5, 20.4){$\mathcal{R}_u$}
		\put(66.5, 88.4){$\mathcal{R}_u$}
		\put(29, 50){$\mathcal{R}_0$}
	\end{overpic}
	\caption{Trajectories of $\Sigma \times \Sigma$ projected on the position plane starting from initial region $\mathcal{R}_0$ (represented by the black triangle).
The regions in red are the unsafe set $\mathcal{R}_u$.}
	\label{fig2}
		\vspace{-0.65cm}
\end{figure}
\textls[-6]{Fig.~\ref{fig2} shows the projection of a few state trajectories on the position plane of the augmented system $\Sigma \times \Sigma$, starting from randomly generated initial conditions in $\mathcal{R}_0$ under control policy $\hat u$ with $u$ taking values in $\mathbb{U}$.  
It is seen that any trajectory starting from $\mathcal{R}_0$ does not reach the unsafe region $\mathcal{R}_u$ as time increases. We further notice that  $\delta = 1$ is the smallest threshold for which we are able to find a barrier certificate ensuring approximate initial-state opacity. For a smaller value of $\delta$, approximate initial-state opacity is immediately violated at the initial condition since the assumption in \eqref{initassum} is not valid anymore.}

\vspace{-0.2cm}
\subsection{Verifying Lack of Approximate Initial-State Opacity on a Room Temperature Model}

In this example, we showcase the use of an ACBC in verifying the lack of opacity in a two-room temperature model by Proposition \ref{BC1}. The model is borrowed from \cite{meyer2017compositional}. The evolution of the temperature $\mathbf{T}(\cdot)$ of 2 rooms is described by the discrete-time model: \\[-13pt]
\begin{align}\label{room}
\Sigma:\left\{
\begin{array}{rl}
\mathbf{T}(k)=& A\mathbf{T}(k) + \alpha_h T_h \nu(k) + \alpha_e T_e,\\
\mathbf{y}(k)=&h(\mathbf{T}(k)),
\end{array}
\right.
\end{align}
\textls[-1]{where $A \!\in\! \mathbb{R}^{2\!\times\! 2}$ is a matrix with elements $\{A\}_{ii} \!=\! (1\! - \!2 \alpha\!-\!\alpha_e\!-\!\alpha_h \nu_i)$, $\{A\}_{12} \!=\! \{A\}_{21} \! = \!\alpha$, $\mathbf{T}(k)\!=\![\mathbf{T}_1(k);$ $\mathbf{T}_2(k)]$,   $T_e\!=\![T_{e1};T_{e2}]$, $\nu(k)\!=\![\nu_1(k);\nu_2(k)]$, where $\nu_i(k)\!\in\! [0,1]$, $\forall i\!\in\![1;2]$, represents the ratio of the heater valve being open in room $i$.
The output of the network  is assumed to be the temperature of the second room: $h(\mathbf{T}(k))\!=\! \mathbf{T}_2(k)$.
Parameters $\alpha\! =\! 0.05$, $\alpha_e\!=\!0.008$, and $\alpha_h\!=\! 0.0036$ are heat exchange coefficients, 
$T_e\! =\! -1\,^\circ C$ is the external temperature, and  $T_h \!=\!50\,^\circ C$ is the heater temperature.  
The regions of interest in this example are $\mathbb{X}\!=\![0, 50]^2$, $\mathbb{X}_{0}\!=\![21, 22]^2$, and $\mathbb{X}_{s}\!=\![21.5, 50] \!\times\! [0, 50]$. Specifically, the secret of the network is whether the first room has a temperature initially higher than $21.5\,^\circ C$ (which may indicate activities with people gathering in that room). 
The intruder wants to infer the initial temperature of the first room by monitoring the temperature variation of the last room and using the knowledge of the system model.
Now the objective is to verify if the system is able to keep this secret in the presence of a malicious intruder with measurement precision $\delta \!=\! 1$. 
In this example, a degree bound of 8 is imposed on $\mathcal{B}$ and $V$.  
First, by means of SOSTOOLS, we failed to find a function $\mathcal{B}(x, \hat x)$ satisfying \eqref{SOSpoly1}-\eqref{SOSpoly3} in Lemma \ref{SOS}. Then, we compute a function
$V(x, \hat x)$ as in Lemma \ref{SOS1} to see if the system is lacking the approximate initial-state opacity.  In this case, the regions considered in Lemma \ref{SOS1} are}  \vspace{-0.2cm}
\begin{align} \notag
&\mathcal{R}_0 \!=\! \{\mathbf{T} \!\in\! [21.5,\!22] \!\times\! [21,\!22],\hat{\mathbf{T}}\!\in\! [21,\!21.5] \!\times\! [21,\!22]\}, \\\notag
&\partial \mathcal{R} \! \setminus\!  \partial  \mathcal{R}_u \! =\! \{(\mathbf{T}, \hat{\mathbf{T}})\!\in\!\mathcal{R}\mid 
(\mathbf{T}_1,\hat{\mathbf{T}}_1) \!\in\! \mathcal{R}_1 \cup \mathcal{R}_2 \!\cup \!\mathcal{R}_3 \!\cup\! \mathcal{R}_4 \}, \\ \notag
&\overline{(\mathcal{R} \setminus \mathcal{R}_u)}
\! =\! \{(\mathbf{T}, \hat{\mathbf{T}})\!\in\! \mathcal{R} \mid (\mathbf{T}_1 \!-\!\hat{\mathbf{T}}_1)^2 \leq \delta^2\},
\end{align}\\[-13pt]
\textls[-1]{where $\mathcal{R} \!=\! \mathbb{X}\times\mathbb{X}$, $\mathcal{R}_1  \!= \! [0, \delta] \!\times\!\{0\}$, $\mathcal{R}_2 \!=\! \{0\} \!\times\![0, \delta]$,  $\mathcal{R}_3 \!=\! \{50\} \!\times\![50\!-\!\delta, 50]$, $\mathcal{R}_4 \!=\![50\!-\!\delta, 50] \!\times\! \{50\}$.
With the aid of SOSTOOLS and SeDuMi, we obtained a polynomial barrier certificate of degree 6 satisfying \eqref{SOSpoly4}-\eqref{SOSpoly7} with a tolerance $\varepsilon\!=\!0.01$ and  
control policy $\nu(k)\! =\! [0;0], \forall k \!\in\! \mathbb{N}_{\geq 0}$. 
The system is thus lacking $1$-approximate initial-state opacity.
This means that for each state run starting from a secret initial state in $\Sigma$ under $\nu$, all trajectories from non-secret states will eventually deviate from the former ones in the sense of generating different outputs (captured by $\delta$). 
Once the intruder sees these trajectories, it is certain that the system was initiated from a secret state. 
Fig.~\ref{fig3} shows trajectories of $\Sigma \!\times\! \Sigma$ from $\mathcal{R}_0$ under control sequence $\nu$ with $\hat \nu$ taking values in $\mathbb{U}$. The trajectories eventually reach $\mathcal{R}_u$ in finite time.}   
\begin{figure}[tb!]
		\vspace{0.2cm}
	\centering
	\begin{overpic}[width=.21\textwidth]{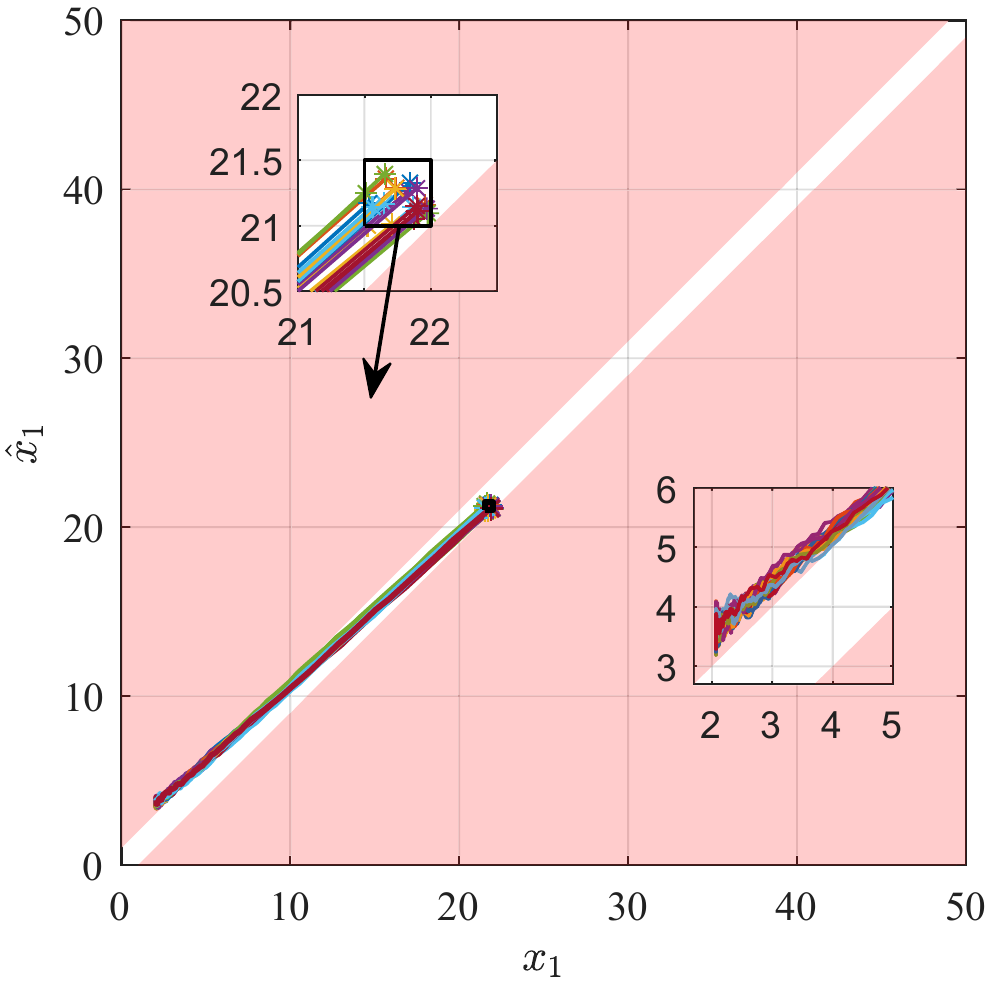}
		\put(30, 20){$\mathcal{R}_u$}
		\put(68, 85){$\mathcal{R}_u$}
		\put(37, 53){$\mathcal{R}_0$}
	\end{overpic}
	\caption{Trajectories of $\Sigma \times \Sigma$ projected on the first-room plane starting from initial region $\mathcal{R}_0$ (represented by the black rectangle). The regions in red are the unsafe set $\mathcal{R}_u$.}
	\label{fig3}
	\vspace{-0.65cm}
\end{figure}
\vspace{-0.15cm}
\section{CONCLUSIONS} \label{Sec:conclusion}
\textls[-5]{We proposed a discretization-free framework for opacity verification of discrete-time control systems.
A pair of augmented control barrier certificates were defined for the analysis of approximate initial-state opacity, which are constructed over an augmented system that is the product of a control system and itself. While both barrier certificates only serve as sufficient conditions, they can be utilized in reverse directions in the sense that one ensures approximate initial-state opacity, and the other one shows the lack of approximate initial-state opacity of the control system. 
We showed that the computation of the barrier certificates can be carried out by some SOS programming. Numerical case studies were conducted to illustrate the effectiveness of the proposed results. Future research will look into the formal synthesis of controllers enforcing approximate opacity properties for control systems using barrier certificates.}

\vspace{-0.25cm}

\addtolength{\textheight}{-12cm}

\bibliographystyle{IEEEtran}
\bibliography{refr}

\end{document}